\documentclass[sn-mathphys,Numbered]{sn-jnl}

\usepackage{graphicx}
\usepackage{multirow}
\usepackage{amsmath,amssymb,amsfonts}
\usepackage{amsthm}
\usepackage{mathrsfs}
\usepackage[title]{appendix}
\usepackage{xcolor}
\usepackage{textcomp}
\usepackage{manyfoot}
\usepackage{booktabs}
\usepackage{algorithm}
\usepackage{algorithmicx}
\usepackage{algpseudocode}
\usepackage{listings}
\usepackage[utf8]{inputenc}

\theoremstyle{thmstyleone}
\newtheorem{theorem}{Theorem}[section]
\newtheorem{proposition}[theorem]{Proposition}
\newtheorem{lemma}[theorem]{Lemma}

\theoremstyle{thmstyletwo}

\newtheorem{remark}[theorem]{Remark}

\theoremstyle{thmstylethree}

\everymath{\displaystyle}

\begin{document}

\title[Generic $\mathcal{A}$-finite determinacy and singularities of homogeneous polynomial mappings]{Generic $\mathcal{A}$-finite determinacy and singularities of homogeneous polynomial mappings}

\author*[1]{\fnm{Nivaldo G.} \sur{Grulha Jr.} }\email{njunior@icmc.usp.br}

\author*[2]{\fnm{João Vítor} \sur{Pissolato} }\email{jvpissolato@usp.br}

\author*[3]{\fnm{Maria A. S.} \sur{Ruas} }\email{maasruas@icmc.usp.br}

\affil*[1,2,3]{\orgdiv{Department of Mathematics}, \orgname{Instituto de Ciências Matemáticas e de Computação}, \orgaddress{\street{Trabalhador São-Carlense Avenue, 400}, \city{São Carlos}, \postcode{13566-590}, \state{São Paulo}, \country{Brazil}}}

\abstract{
We make a detailed investigation of the generic properties that polynomial mappings possess. An important starting point is the work by Farnik, Jelonek and Ruas in 2019, where they prove some of those properties in the context of homogeneous polynomial mappings of $\mathbb{C}^3$ to $\mathbb{C}^3$, and conclude the genericity of $\mathcal{A}$-finite determinacy by applying the geometric criterion. Using their strategy, we further extend and generalize some of their key findings to dimensions greater than or equal to $2$, though some of those properties can only be extended up to dimension $4$.}


\maketitle	

\section{Introduction}\label{sec1}

In 1971, Mather obtained an important result about certain pairs of dimensions $(n, p)$, now referred to as Mather's nice dimensions. It is a well-known fact that the set of proper $\mathcal{A}$-stable mappings $f:\mathbb{R}^n\to \mathbb{R}^p$ is dense in the space of proper mappings equipped with the Whitney Topology when $(n,p)$ is in the set of nice dimensions. However, it is important to note that not all pairs of dimensions have this property. For example, for $n=p$, the pair $(n, n)$ does not belong to the set of the nice dimensions if $n\geq9$, see \cite{11} and \cite{12} for details.



In the complex polynomial mappings case, we refer the works made by Farnik, Jelonek and Ruas in \cite{10} and \cite{1}, where they study the set of mappings $F:\mathbb{C}^n\to\mathbb{C}^n$, for $n=2,3$, using Algebraic Geometry techniques and prove many generic properties that such mappings satisfy. We extend their results for higher dimensions, though some properties cannot be extended to every $n$.


We follow the notation given in \cite{1}. Let $d_i\geq 2$ be integers, for $1\leq i\leq n$. We denote by $\Omega(d_1,\ldots,d_n)$ the space of all mappings $F:\mathbb{C}^n\to\mathbb{C}^n$, where $F=(f_1,\ldots,f_n)$, and each $f_i$ is a polynomial in variables $x_1,\ldots,x_n$ of degree less than or equal to $d_i$. Additionally, we introduce the subspace $H(d_1,\ldots,d_n)$ of $\Omega(d_1,\ldots,d_n)$, which consists of all mappings $F=(f_1,\ldots,f_n)$, where each $f_i$ is a homogeneous polynomial of degree $d_i$. More precisely, each $f_k$ can be written as \begin{equation}\label{exp1}f_k=\sum_{i_2,\ldots,i_n}a_{i_2,\ldots,i_n;k}x_1^{d_k-i_2-\cdots-i_n}x_2^{i_2}\cdots x_n^{i_n}.\end{equation}

Both spaces $\Omega(d_1,\ldots,d_n)$ and $H(d_1,\ldots,d_n)$ can be identified with the affine space $\mathbb{C}^N$, for a large enough $N$, by considering the coefficients of each polynomial as a coordinate. Therefore, those spaces can be equipped with the Zariski Topology, and a closed set in those spaces can be characterized as the common zeroes of a finite number of polynomial relations involving those coefficients. Since the affine space is irreducible, then every non-empty open Zariski set is dense. Consequently, we can define a property in $H(d_1,\ldots,d_n)$ as \textit{generic} if it holds for every polynomial in such a set. Furthermore, due to the irreducibility of the affine space, the intersection of finitely many non-empty open sets is also a non-empty open set. Hence, an intersection of finitely many generic properties is a generic property.

\section{Basic Results}


	Let $\mathbb{C}[x_1,\ldots,x_n]$ denote the polynomial ring in $n$ variables $x_1,\ldots,x_n$ with complex coefficients. Given an ideal $I$ of $\mathbb{C}[x_1,\ldots,x_n]$, we define $\text{Z}(I)$ as  $$\text{Z}(I)=\{x\in\mathbb{C}^n\mid f(x)=0,\text{ }\forall f\in I\},$$ that is, $\text{Z}(I)$ is the set of the common solutions in $\mathbb{C}^n$ to the system of equations $f(x)=0$, for all $f\in I$. 
	
	In $\mathbb{C}^n$, we introduce the Zariski Topology, defining a subset $X\subseteq\mathbb{C}^n$ as closed if there exists an ideal $I$ of $\mathbb{C}[x_1,\ldots,x_n]$ such that $X=\text{Z}(I)$. An algebraic set is a subset that is closed in this topology.
	
	Given two affine spaces $\mathbb{C}^n$ and $\mathbb{C}^m$, we can define the Zariski Topology in the cartesian product $\mathbb{C}^n\times \mathbb{C}^m$, defining a closed set as the set of common zeroes of an ideal in the polynomial ring $\mathbb{C}[x_1,\ldots,x_n,y_1,\ldots,y_m]$.
	
	An algebraic variety in $\mathbb{C}^n$ is an irreducible algebraic set, that is, a closed set that cannot be written as an union of two proper closed sets. Since the polynomial ring is a Noetherian ring, every algebraic set can be written as a finite union of irreducible closed sets
	
	The dimension of an algebraic variety is defined as the maximum length of increasing chains of closed subsets. More generally, the dimension of an algebraic set is the maximum dimension of its irreducible components.
	
	A morphism between two algebraic sets is a mapping $f:X\to Y$, where $X\subseteq\mathbb{C}^n$ and $Y\subseteq\mathbb{C}^m$ are algebraic sets, such that there exists polynomials $f_1,\ldots,f_m\in\mathbb{C}[x_1,\ldots,x_n]$ satisfying $f(x)=(f_1(x),\ldots,f_m(x))$ for all $x\in X$. The morphism $f$ is said to be dominant if its image $f(X)$ is dense in $Y$. An isomorphism is a bijective morphism where the inverse mapping is also a morphism.
	
	Given $F\in\Omega(d_1,\ldots,d_n)$, we denote by $\text{C}(F)$ the algebraic set $\text{Z}(J(F))$, where $J(F)$ is the determinant of the Jacobian matrix of $F$. The set $C(F)$ is the critical set of $F$, which consists of points where the derivative of $F$ is not surjective. Furthermore, we denote by $\Delta(F)$ the set $F(\text{C}(F))$, which is the discriminant of $F$.
	

\section{The main results}\label{sec2}

Our first main result is a compilation of several lemmas, which we prove in the following section. In \cite{1}, the authors prove some properties that are generic in the space of homogeneous polynomial mappings $F:\mathbb{C}^3\to\mathbb{C}^3$ of degrees $(d_1,d_2,d_3)$. We extend their arguments for $F:\mathbb{C}^n\to\mathbb{C}^n$, for $n\geq2$, with some adaptations, though some of those lemmas are true only up to $n=4$.

Theorem \ref{teo1} and \ref{teo5} give properties that hold for sufficiently generic homogeneous polynomial mappings of bounded degrees. Our next result is Theorem \ref{teo6}, which is a consequence of those two theorems, showing that such generic mappings are $\mathcal{A}$-finitely determined, which follows by applying the geometric criterion for $\mathcal{A}$-finite determinancy, see \cite{5} for reference.

Furthermore, we can apply the global techniques proved by Ohmoto in \cite{7} to count the $0$-stable invariants of a generic polynomial mapping $F:\mathbb{C}^4\to\mathbb{C}^4$.

\begin{theorem}\label{teo1}
	
Assume that $\text{gcd}(d_i,d_j)\leq2$ and $\text{gcd}(d_i,d_j,d_k)=1$, for $1\leq i< j<k\leq n$. Then there exists a dense Zariski open set $U\subseteq H(d_1,\ldots,d_n)$ such that, for every $F\in U$, the following properties holds for $F$: 	

\begin{enumerate}
	
\item $F^{-1}(0)=\{0\}$;

\item If $d_i$, for $1\leq i\leq n$, are pairwise coprime, then $F$ restricted to any ray contained in $\text{C}(F)$ is injective. If a pair is not coprime, that is, $d_{i_1}$ and $d_{i_2}$ are even and  $d_{i_3},\ldots,d_{i_n}$ are odd, then $F$ restricted to any ray contained in $\text{C}(F)\setminus\text{Z}(f_{i_j})$ is injective, for $3\leq j\leq n$, and $F$ restricted to any of the finite number of rays contained in $\text{C}(F)\cap\text{Z}(f_{i_3},\ldots,f_{i_n})$ is two-to-one.

\item $F\big|_{\text{C}(F)}$ is injective outside of a set of dimension less than or equal to $n-2$ in $\mathbb{C}^n$.

\item If $p\in\Delta(F)$, then $|F^{-1}(p)\cap\text{C}(F)|\leq n-1$.
	
\end{enumerate}
	
\end{theorem}

\begin{theorem}\label{teo5}

There exists a dense Zariski open set $U\subseteq H(d_1,d_2,d_3,d_4)$ such that, for every $F\in U$, the following properties holds for $F$:

\begin{enumerate}

\item Outside the origin, the singularities of $F$ are only $A_1$ (fold), $A_2$ (cusp) and $A_3$ (swallowtail). In particular, $\text{C}(F)\setminus\{0\}$ is smooth.

\item If $F$ has a swallowtail in $p$, then $F^{-1}(F(p))\cap \text{C}(F)=\{p\}$.

\item If $|F^{-1}(p)\cap\text{C}(F)|\geq2$, then the hypersurface $\Delta(F)$ has a normal crossing at $p$.

\end{enumerate}
	
\end{theorem}

\begin{theorem}\label{teo6}
	If $\text{gcd}(d_i,d_j)\leq 2$, for $1\leq i<j\leq 4$, and $\text{gcd}(d_i,d_j,d_k)=1$, for $1\leq i<j<k\leq 4$, then there exists a dense Zariski open set $U\subseteq H(d_1,d_2,d_3,d_4)$ such that, for every $F\in U$, we have that the germ $F:(\mathbb{C}^4,0)\to(\mathbb{C}^4,0)$ is $\mathcal{A}$-finitely determined.
	
	Moreover, if $\text{gcd}(d_1,d_2,d_3,d_4)>1$, then there are no homogeneous polynomial mapping germ of degrees $d_1,d_2,d_3,d_4$ that is $\mathcal{A}$-finitely determined.
\end{theorem}

\section{Results and proofs}\label{sec3}

The strategy to prove our main results are based on the application of Algebraic Geometry techniques. For the sake of completeness, we provide a detailed review of essential definitions, notations, and important facts for the following proofs.

\begin{remark}


By the Hilbert's Basis Theorem, the ideal $I$ is finitely generated. If we denote this ideal as $I=\langle f_1,\ldots,f_k\rangle$, then the zero set $\text{Z}(I)$ is denoted by $$\text{Z}(f_1,\ldots,f_k)=\{x\in\mathbb{C}^n\mid f_1(x)=\cdots=f_k(x)=0\}.$$ Let $F\in H(d_1,\ldots,d_n)$. Since each $f_i$ is homogeneous, their partial derivatives are also homogeneous, implying that the jacobian determinant $J(F)$ is a homogeneous polynomial. Consequently, $\text{C}(F)$ is composed of lines passing through the origin in $\mathbb{C}^n$.

Since we are in the equidimensional case, the critical set of $F$ coincides with the singular set of $F$, where the singular set is the set of points where the differential of $F$ fails to be both injective and surjective.

\end{remark}

We frequently apply the Theorem of Dimension of the Fibres, which can be found in \cite[4.4 Theorem on the Dimension of the Fibres, p.228]{2}, to prove our results. Herein, we state the theorem for reference:

\begin{theorem}[Dimension of the Fibres]\label{tdf}
Let $X$ and $Y$ be two algebraic sets, and let $f:X\to Y$ be a dominant morphism (particularly, when $f$ is surjective). Then, for all $x\in X$, we have $$\text{dim}(f^{-1}(f(x)))\geq \text{dim}(X)-\text{dim}(Y).$$ Furthermore, there exists a dense Zariski open set $U\subseteq Y$ such that, for every $y\in U$, the following equality holds: $$\text{dim}(f^{-1}(y))=\text{dim}(X)-\text{dim}(Y).$$
\end{theorem}

With this theorem, we are now in position to prove our first main result. In order to do so, it is sufficient to obtain a non-empty dense Zariski open set for each item and then take $U$ as their intersection. Hence, we state and prove each item as an individual lemma. The proof strategy is similar for all the lemmas, and thus we show all the details for the first lemma while omitting the analogous arguments for the other ones.

\begin{lemma}\label{lemma1}
There exists a dense Zariski open set $U\subseteq H(d_1,\ldots,d_n)$ such that, for every $F\in U$, we have $F^{-1}(0)=\{0\}$.
\end{lemma}
\begin{proof}

Define $$X=\{(p,F)\in\mathbb{C}^n\times H(d_1,\ldots,d_n)\mid F(p)=0\}$$ and consider the projections $\pi_1:X\to\mathbb{C}^n$, $\pi_2:X\to H(d_1,\ldots,d_n)$. By Theorem \ref{tdf}, there exists a non-empty Zariski open subset $U_1\subseteq\mathbb{C}^n$ such that $$\text{dim}(\pi_1^{-1}(x))=\text{dim}(X)-\text{dim}(\mathbb{C}^n)=\text{dim}(X)-n,\text{ }\forall x\in U_1.$$ Consider the non-empty Zariski open subset $U_2=\mathbb{C}^n\setminus\{0\}$. Given $x\in U_2$, we consider a linear isomorphism $T$ such that $x\mapsto e_1$. It follows that $T$ induces an isomorphism $\pi_1^{-1}(x)\cong \pi_1^{-1}(e_1)$, where each $(x,F)\in\pi_1^{-1}(x)$ is identified with $(e_1,F\circ T^{-1})\in\pi_1^{-1}(e_1)$.

Let us describe and compute the dimension of the fiber $\pi_1^{-1}(e_1)$, which we identify as a subset of $H(d_1,\ldots,d_n)$. Note that, from \ref{exp1}, we have $f_k(e_1)=a_{0,\ldots,0;k}$, for $1\leq k\leq n$. Then $$\pi_1^{-1}(e_1)\cong\{F\in H(d_1,\ldots,d_n)\mid a_{0,\ldots,0;1}=\cdots=a_{0,\ldots,0;n}=0\},$$ which has codimension $n$ in $H(d_1,\ldots,d_n)$. Now, take a point $x\in U_1\cap U_2$. Since $\pi_1^{-1}(x)\cong\pi_1^{-1}(e_1)$, it follows that their dimensions are equal, that is, $$\text{dim}(X)-n=\text{dim}(H(d_1,\ldots,d_n))-n\Longrightarrow \text{dim}(X)-\text{dim}(H(d_1,\ldots,d_n))=0.$$ On the other hand, applying Theorem \ref{tdf} to $\pi_2$, we obtain a non-empty Zariski open subset $U\subseteq H(d_1,\ldots,d_n)$ such that, for every $F\in U$, we have $$\text{dim}(\pi_2^{-1}(F))=\text{dim}(X)-\text{dim}(H(d_1,\ldots,d_n))=0.$$ Let us prove that this $U$ satisfies the property given in the lemma. Take $F\in U$. Since $F$ is homogeneous, we have that $F(0)=0$, so $\{0\}\subseteq F^{-1}(0)$. On the other hand, if we suppose that there exists $p\in F^{-1}(0)\setminus\{0\}$, since $\pi_2^{-1}(F)$ is given by a homogeneous equation, we have that $$(\lambda p,F)\in\pi_2^{-1}(F),\text{ }\forall\lambda\in\mathbb{C}.$$ Thus, the fiber $\pi_2^{-1}(F)$ contains a line, which contradicts the dimension of this fiber. Therefore, we conclude that $F^{-1}(0)=\{0\}$, for every $F\in U$.

\end{proof}

\begin{remark}

The set $F^{-1}(0)$ is the set of points that are solutions to a system of $n$ homogeneous equations, represented by $f_k(x_1, \ldots, x_n) = 0$ for $1 \leq k \leq n$, in the complex space $\mathbb{C}^n$. Each equation defines a hypersurface consisting of lines passing through the origin. The above lemma says that, in a generic choices of such equations, the intersection of $n$ such hypersurfaces results in only the origin as a common point. Furthermore, we show the following proposition, which states the topological equivalence of this geometric interpretation.

\end{remark}

\begin{proposition}\label{prop1}
Let $F\in H(d_1,\ldots,d_n)$. Then $F:\mathbb{C}^n\to\mathbb{C}^n$ is proper (with respect to the usual topology) if, and only if, $F^{-1}(0)=\{0\}$.
\end{proposition}

\begin{proof}
Suppose that $F$ is proper. In particular, we have that $F^{-1}(0)$ is compact (in fact, it is a set of finite points, which can be shown by applying the Geometric Form of Noether's Normalization Theorem, given in \cite[p. 42]{3}, to every irreducible component of $F^{-1}(0)$) and, in particular, it is bounded. By homogeneity, we have $\{0\}\subseteq F^{-1}(0)$, and if we suppose that there exists a point $p\in F^{-1}(0)\setminus\{0\}$, we would have that $\lambda p\subseteq F^{-1}(0)$, for every $\lambda\in\mathbb{C}$, which is a contradiction with the boundedness of $F^{-1}(0)$. Therefore, $F^{-1}(0)=\{0\}$.

Conversely, suppose that $F^{-1}(0)=\{0\}$. In order to show that $F$ is proper, it suffices to prove that if $(x_k)_{k\in\mathbb{N}}$ is a sequence of points in $\mathbb{C}^n$ such that $x_k\to\infty$, we have that $F(x_k)\to\infty$. Since $F^{-1}(0)=\{0\}$, we have that $0\notin||F(S^{2n-1})||$, where $S^{2n-1}$ denotes the unit sphere in $\mathbb{C}^n$. Since $\mathbb{R}$ is a regular space, we can separate $0$ from the closed set $||F(S^{2n-1})||$, that is, there exists $\alpha>0$ such that $$||F(x)||\geq\alpha>0,\text{ }\forall x\in S^{2n-1}.$$ Since $x_k\to\infty$, we can suppose that $||x_k||\geq 1$, for all $k$, and it follows that $||F(x_k)||\geq||x_k||^d\alpha$, for all $k$, where $d=\text{min}\{d_1,\ldots,d_n\}$. Hence, $F(x_k)\to\infty$. Therefore, $F$ is proper.
\end{proof}

\begin{remark}\label{rmk4}

The proposition above shows the existence of a non-empty Zariski open subset $U\subset H(d_1,\ldots,d_n)$ such that every $F\in U$ is proper. Now we consider the continuous mapping $$\pi:\Omega(d_1,\ldots,d_n)\to H(d_1,\ldots,d_n)$$ given by the projection of the coefficients of the homogeneous part of highest degree of $F$. We observe that the condition $F(x_k)\to\infty$ is determined by this highest degree homogeneous part of $F$. Thus, we conclude that $\pi^{-1}(U)$ is also a non-empty Zariski open subset of $\Omega(d_1,\ldots,d_n)$, and every $F\in\pi^{-1}(U)$ is proper. In other words, a generic polynomial mapping $F:\mathbb{C}^n\to\mathbb{C}^n$ is proper.

\end{remark}

\begin{lemma}\label{lemma2}
	
Assume $\text{gcd}(d_i,d_j)\leq2$ and $\text{gcd}(d_i,d_j,d_k)=1$, for $1\leq i<j<k\leq n$. If $d_i$, for $1\leq i\leq n$, are pairwise coprime, then there exists a dense Zariski open set $U\subseteq H(d_1,\ldots,d_n)$ such that, for every $F\in U$, $F$ restricted to any ray contained in $\text{C}(F)$ is injective.

If a pair is not coprime, that is, $d_{i_1}$ and $d_{i_2}$ are even, and $d_{i_3},\ldots,d_{i_n}$ are odd, then $F$ restricted to any ray contained in $\text{C}(F)\setminus\text{Z}(f_{i_j})$ is injective, for $3\leq j\leq n$, and $F$ restricted to any of the finite number of rays contained in $\text{C}(F)\cap\text{Z}(f_{i_3},\ldots,f_{i_n})$ is two-to-one.

\end{lemma}

\begin{proof}

First, let us study the injectivity of the mapping $F\big|_{\mathbb{C}\cdot p}$, where $$\mathbb{C}\cdot p=\{\lambda p\mid \lambda\in\mathbb{C}\}\subseteq\text{C}(F)$$ is a complex line passing through the origin and $p$. Given $x,y\in\mathbb{C}\cdot p$, with $x=\lambda p$ and $y=\mu p$, we have that $$\begin{array}{rcl}F(x)=F(y)&\Rightarrow& f_i(\lambda p)=f_i(\mu p),\text{ }1\leq i\leq n\\&\Rightarrow& \lambda^{d_i}f_i(p)=\mu^{d_i}f_i(p),\text{ }1\leq i\leq n.\end{array}$$ If we suppose that $f_i(p)\neq 0$ and $f_j(p)\neq 0$, we have from the above equality that $$\left(\frac{\lambda}{\mu}\right)^{d_i}=\left(\frac{\lambda}{\mu}\right)^{d_j}=1.$$ From the Bézout identity we conclude that $$\left(\frac{\lambda}{\mu}\right)^{\text{gcd}(d_i,d_j)}=1.$$ Therefore, if $\text{gcd}(d_i,d_j)=1$, we get $\lambda=\mu$ and $x=y$, proving that $F\big|_{\mathbb{C}\cdot p}$ is injective, and if $\text{gcd}(d_i,d_j)=2$, we get $\lambda=\pm \mu$ and $x=\pm y$, proving that $F\big|_{\mathbb{C}\cdot p}$ can be two-to-one, and this case will be investigated further in this proof. Therefore, it suffices to prove that, generically, given $p\in\text{C}(F)\setminus\{0\}$, there are $i,j\in\{1,\ldots,n\}$ such that $f_i(p)\neq 0$ and $f_j(p)\neq 0$. For this, it suffices to prove that, for generic $F$, we have $$\text{C}(F)\cap\text{Z}(f_{j_1},\ldots,f_{j_{n-1}})=\{0\},$$ for every subset $\{j_1,\ldots,j_{n-1}\}\subseteq\{1,\ldots,n\}$. Indeed, if we prove this claim, since $p\neq 0$, then, by the above argument, we have $p\notin\text{Z}(f_1,\ldots,f_{n-1})$, therefore, there exists $i$ such that $f_i(p)\neq0$. Now, we also have that $p\notin\text{Z}(f_1,\ldots,f_{i-1},f_{i+1},\ldots,f_n)$, then $f_j(p)\neq 0$ for some $j\neq i$, as we wanted. It suffices to prove for one such subset, as the others are analogous, and then take the intersection of the finitely many dense Zariski open sets in $H(d_1,\ldots,d_n)$ that we obtain. For simplicity, let us prove that $$\text{C}(F)\cap\text{Z}(f_1,\ldots,f_{n-1})=\{0\}$$ is a generic condition. Define $$X=\{(p,F)\in\mathbb{C}^n\times H(d_1,\ldots,d_n)\mid J(F)(p)=f_1(p)=\cdots=f_{n-1}(p)=0\}.$$ Once again, it suffices to obtain the dimension of the fiber $\pi_1^{-1}(e_1)$. We have that the equation $f_k(e_1)=0$ is $a_{0,\ldots,0;k}=0$, for $1\leq k\leq n-1$. Now, the equation $J(F)(e_1)=0$ is $$\left|\begin{array}{ccccc} d_1a_{0,0\ldots,0;1}&a_{1,0,\ldots,0;1}&a_{0,1,\ldots,0;1}&\cdots&a_{0,0,\ldots,1;1}\\d_2a_{0,0\ldots,0;2}&a_{1,0,\ldots,0;2}&a_{0,1,\ldots,0;2}&\cdots&a_{0,0,\ldots,1;2}\\\vdots&\vdots&\vdots&\ddots&\vdots\\d_{n-1}a_{0,0\ldots,0;n-1}&a_{1,0,\ldots,0;n-1}&a_{0,1,\ldots,0;n-1}&\cdots&a_{0,0,\ldots,1;n-1}\\d_na_{0,0\ldots,0;n}&a_{1,0,\ldots,0;n}&a_{0,1,\ldots,0;n}&\cdots&a_{0,0,\ldots,1;n}\end{array}\right|=0,$$ and using the $n-1$ equations obtained above, it simplifies to $$\begin{array}{rcl}0&=&\left|\begin{array}{ccccc} 0&a_{1,0,\ldots,0;1}&a_{0,1,\ldots,0;1}&\cdots&a_{0,0,\ldots,1;1}\\0&a_{1,0,\ldots,0;2}&a_{0,1,\ldots,0;2}&\cdots&a_{0,0,\ldots,1;2}\\\vdots&\vdots&\vdots&\ddots&\vdots\\0&a_{1,0,\ldots,0;n-1}&a_{0,1,\ldots,0;n-1}&\cdots&a_{0,0,\ldots,1;n-1}\\d_na_{0,0\ldots,0;n}&a_{1,0,\ldots,0;n}&a_{0,1,\ldots,0;n}&\cdots&a_{0,0,\ldots,1;n}\end{array}\right|\\&=&(-1)^{n+1}d_na_{0,0,\ldots,0;n}\left|\begin{array}{cccc}a_{1,0,\ldots,0;1}&a_{0,1,\ldots,0;1}&\cdots&a_{0,0,\ldots,1;1}\\a_{1,0,\ldots,0;2}&a_{0,1,\ldots,0;2}&\cdots&a_{0,0,\ldots,1;2}\\\vdots&\vdots&\ddots&\vdots\\a_{1,0,\ldots,0;n-1}&a_{0,1,\ldots,0;n-1}&\cdots&a_{0,0,\ldots,1;n-1}\end{array}\right|,\end{array}$$ which is a non-trivial equation. Therefore $$\text{dim}(\pi^{-1}(e_1))=\text{dim}(H(d_1,\ldots,d_n))-n.$$ On the other hand, we have from Theorem \ref{tdf} that $$\text{dim}(\pi^{-1}(x))=\text{dim}(X)-n,$$ for generic $x\in\mathbb{C}^n$. Since these fibers are isomorphic, we conclude that $$\text{dim}(X)-\text{dim}(H(d_1,\ldots,d_n))=0.$$ Now, by Theorem \ref{tdf}, we have that $$\text{dim}(\pi_2^{-1}(F))=\text{dim}(X)-\text{dim}(H(d_1,\ldots,d_n))=0,$$ for generic $F\in H(d_1,\ldots,d_n)$. Let us verify that those $F$ satisfy the wanted property. First, by homogeneity, we have the inclusion $$\{0\}\subseteq\text{C}(F)\cap\text{Z}(f_1,\ldots,f_{n-1}),$$ and if we suppose that there exists $p\in\text{C}(F)\cap\text{Z}(f_1,\ldots,f_{n-1})\setminus\{0\}$, then $(p,F)\in\pi_2^{-1}(F)$, and by the homogeneity of the equations that defines this fiber, we would have that this fiber contains a line, which is a contradiction with its dimension. Therefore, $$\text{C}(F)\cap\text{Z}(f_1,\ldots,f_{n-1})=\{0\},$$ as we wanted.

Now, if a pair of degrees is not coprime, for example, $\text{gcd}(d_1,d_2)=2$, if we take a ray $\mathbb{C}\cdot p\subseteq\text{C}(F)\setminus\text{Z}(f_k)$, for $3\leq k\leq n$, we have that the other degrees $d_i$ are odd. As we have shown at the start of this proof, we have that if $x,y\in\mathbb{C}\cdot p$ are two points such that $F(x)=F(y)$, then $y=\pm x$, but if we had $y=-x$, then $$f_k(x)=f_k(y)=f_k(-x)=(-1)^{d_k}f_k(x)=-f_k(x),$$ which implies that $f_k(x)=0$, which is a contradiction, since $x\notin\text{Z}(f_k)$. Therefore, $F$ is injective for every ray $\mathbb{C}\cdot p\subseteq\text{C}(F)\setminus\text{Z}(f_i)$, for $3\leq i\leq n$.

Now, for any ray $\mathbb{C}\cdot p\subseteq\text{C}(F)\cap\text{Z}(f_3,\ldots,f_n)$, we have that $F$ is two-to-one, since $y=\pm x$. It remains to show that $\text{C}(F)\cap\text{Z}(f_3,\ldots,f_n)$ is a finite number of lines. To prove this, it suffices to show that, for a generic $F$, we have $$\text{dim}(\text{C}(F)\cap\text{Z}(f_3,\ldots,f_n))=1.$$ Define $$X=\{(p,F)\in\mathbb{C}^n\times H(d_1,\ldots,d_n)\mid J(F)(p)=f_3(p)=\cdots=f_n(p)=0\}.$$ As above, we have that $\pi_1^{-1}(e_1)$ is given by the equations $$a_{0,0,\ldots,0;3}=\cdots=a_{0,0,\ldots,0;n}=0$$ and $$\left|\begin{array}{ccccc} d_1a_{0,0,\ldots,0;1}&a_{1,0,\ldots,0;1}&a_{0,1,\ldots,0;1}&\cdots&a_{0,0,\ldots,1;1}\\d_2a_{0,0,\ldots,0;2}&a_{1,0,\ldots,0;2}&a_{0,1,\ldots,0;2}&\cdots&a_{0,0,\ldots,1;2}\\0&a_{1,0,\ldots,0;3}&a_{0,1,\ldots,0;3}&\cdots&a_{0,0,\ldots,1;3}\\\vdots&\vdots&\vdots&\ddots&\vdots\\0&a_{1,0,\ldots,0;n-1}&a_{0,1,\ldots,0;n-1}&\cdots&a_{0,0,\ldots,1;n-1}\\0&a_{1,0,\ldots,0;n}&a_{0,1,\ldots,0;n}&\cdots&a_{0,0,\ldots,1;n}\end{array}\right|=0,$$ which are $n-1$ independent equations, then $$\text{dim}(\pi_1^{-1}(e_1))=\text{dim}(H(d_1,\ldots,d_n))-(n-1).$$ By Theorem \ref{tdf}, we have that a generic fiber has dimension $\text{dim}(X)-n$, so we have that $$\text{dim}(H(d_1,\ldots,d_n))-n+1=\text{dim}(X)-n\Rightarrow \text{dim}(X)-\text{dim}(H(d_1,\ldots,d_n))=1.$$ Again, from Theorem \ref{tdf}, we have $$\text{dim}(\pi_2^{-1}(F))=\text{dim}(X)-\text{dim}(H(d_1,\ldots,d_n))=1,$$ for a generic $F\in H(d_1,\ldots,d_n)$, now we observe that this fiber is isomorphic to $\text{C}(F)\cap\text{Z}(f_3,\ldots,f_n)$, and the lemma follows.

\end{proof}

\begin{remark}\label{rmk1}
The hypothesis of the previous lemma cannot be omitted. For instance, if we suppose that $d=\text{gcd}(d_1,d_2)>2$, then given a ray $\mathbb{C}\cdot p\subseteq\text{C}(F)\cap\text{Z}(f_3,\ldots,f_n)$, and $\xi\in\mathbb{C}$ a $d$-th primitive root of the unity, we obtain $$F(\xi^ip)=((\xi^i)^{d_1}f_1(p),(\xi^i)^{d_2}f_2(p),\ldots,(\xi^i)^{d_n}f_n(p))=(f_1(p),f_2(p),\ldots,f_n(p))=F(p),$$ for $0\leq i\leq d-1$. This is due to the fact that both $d_1$ and $d_2$ are multiples of $d$, and $f_3,\ldots,f_n$ vanish at $p$. Consequently, $F$ is a $d$-to-one mapping when restricted to any ray contained in $\text{C}(F)\cap\text{Z}(f_3,\ldots,f_n)$.

Moreover, in the case where $d=\text{gcd}(d_1,\ldots,d_n)>1$, considering any ray $\mathbb{C}\cdot p\subseteq\text{C}(F)$ and $\xi$ as a $d$-th primitive root of unity, we obtain \begin{align*}
F(\xi^ip)&=(\xi^{id_1}f_1(p),\ldots,\xi^{id_n}f_n(p))\\
&=(f_1(p),\ldots,f_n(p))\\
&=F(p),
\end{align*} for $0\leq i\leq d-1$, which shows that $F$ is not injective in any ray contained in $\text{C}(F)$.
	
\end{remark}

\begin{lemma}\label{lemma3}
	
There exists a dense Zariski open set $U\subseteq H(d_1,\ldots,d_n)$ such that, for every $F\in U$, we have that $F\big|_{\text{C}(F)}$ is injective outside a set of dimension less than or equal to $n-2$ in $\mathbb{C}^n$.

\end{lemma}

\begin{proof}
	
Define $$X=\{(p_1,p_2,F)\in(\mathbb{C}^n)^2\times H(d_1,\ldots,d_n)\mid F(p_1)=F(p_2),\text{ }J(F)(p_1)=J(F)(p_2)=0\}.$$ Observe that, generically, two vectors in $\mathbb{C}^n$ are linearly independent, since the condition that a set $\{x,y\}\subseteq\mathbb{C}^n$ is linearly dependent defines a proper Zariski closed subset in $\mathbb{C}^n\times\mathbb{C}^n$. Thus, the complement of this set is a dense open Zariski subset. If we take a pair $(p_1,p_2)$ in this set, we can consider the linear isomorphism obtained by linear extension such that $p_i\mapsto e_i$, for $i=1,2$. Therefore, it suffices to obtain the dimension of the fiber $\pi_1^{-1}(e_1,e_2)$. The equation $F(e_1)=F(e_2)$ is $$a_{0,0,\ldots,0;k}=a_{d_k,0,\ldots,0;k},\text{ for }1\leq k\leq n,$$ which describes a subspace of codimension $n$. Now, the equation $J(F)(e_1)=0$ is $$\left|\begin{array}{ccccc}d_1a_{0,0,\ldots,0;1}&a_{1,0,\ldots,0;1}&a_{0,1,\ldots,0;1}&\cdots&a_{0,0,\ldots,1;1}\\d_2a_{0,0,\ldots,0;2}&a_{1,0,\ldots,0;2}&a_{0,1,\ldots,0;2}&\cdots&a_{0,0,\ldots,1;2}\\\vdots&\vdots&\vdots&\ddots&\vdots\\d_na_{0,0,\ldots,0;n}&a_{1,0,\ldots,0;n}&a_{0,1,\ldots,0;n}&\cdots&a_{0,0,\ldots,1;n}\end{array}\right|=0$$ and the equation $J(F)(e_2)=0$ is $$\left|\begin{array}{cccccc}a_{d_1-1,0,0,\ldots,0;1}&d_1a_{d_1,0,0,\ldots,0;1}&a_{d_1-1,1,0,\ldots,0;1}&a_{d_1-1,0,1,\ldots,0;1}&\cdots&a_{d_1-1,0,0,\ldots,1;1}\\a_{d_2-1,0,0,\ldots,0;2}&d_2a_{d_2,0,0,\ldots,0;2}&a_{d_2-1,1,0,\ldots,0;2}&a_{d_2-1,0,1,\ldots,0;2}&\cdots&a_{d_2-1,0,0,\ldots,1;2}\\\vdots&\vdots&\vdots&\vdots&\ddots&\vdots\\a_{d_n-1,0,0,\ldots,0;n}&d_na_{d_n,0,0,\ldots,0;n}&a_{d_n-1,1,0,\ldots,0;n}&a_{d_n-1,0,1,\ldots,0;n}&\cdots&a_{d_n-1,0,0,\ldots,1;n}\end{array}\right|=0.$$ Note that these two equations, together with the equation $F(e_1)=F(e_2)$, become $$\left|\begin{array}{ccccc}d_1a_{0,0,\ldots,0;1}&a_{1,0,\ldots,0;1}&a_{0,1,\ldots,0;1}&\cdots&a_{0,0,\ldots,1;1}\\d_2a_{0,0,\ldots,0;2}&a_{1,0,\ldots,0;2}&a_{0,1,\ldots,0;2}&\cdots&a_{0,0,\ldots,1;2}\\\vdots&\vdots&\vdots&\ddots&\vdots\\d_na_{0,0,\ldots,0;n}&a_{1,0,\ldots,0;n}&a_{0,1,\ldots,0;n}&\cdots&a_{0,0,\ldots,1;n}\end{array}\right|=0$$ and $$\left|\begin{array}{cccccc}a_{d_1-1,0,0,\ldots,0;1}&d_1a_{0,0,0,\ldots,0;1}&a_{d_1-1,1,0,\ldots,0;1}&a_{d_1-1,0,1,\ldots,0;1}&\cdots&a_{d_1-1,0,0,\ldots,1;1}\\a_{d_2-1,0,0,\ldots,0;2}&d_2a_{0,0,0,\ldots,0;2}&a_{d_2-1,1,0,\ldots,0;2}&a_{d_2-1,0,1,\ldots,0;2}&\cdots&a_{d_2-1,0,0,\ldots,1;2}\\\vdots&\vdots&\vdots&\vdots&\ddots&\vdots\\a_{d_n-1,0,0,\ldots,0;n}&d_na_{0,0,0,\ldots,0;n}&a_{d_n-1,1,0,\ldots,0;n}&a_{d_n-1,0,1,\ldots,0;n}&\cdots&a_{d_n-1,0,0,\ldots,1;n}\end{array}\right|=0.$$ We can expand the first determinant with respect to the first column and the second determinant with respect the second column to rewrite these two equations as $$\sum_{k=1}^nd_ka_{0,0,\ldots,0;k}m_{k,1}=0=\sum_{k=1}^nd_ka_{0,0,\ldots,0;k}m_{k,2},$$ where $m_{k,i}$ are the respective minors. Now, outside of each $\text{Z}(a_{0,0,\ldots,0;k})$ we have that these two equations are independent, since $m_{k,1}$ and $m_{k,2}$ are given by different sets of variables. So, outside of $\text{Z}(a_{0,0,\ldots,0;1},\ldots,a_{0,0,\ldots,0;n})$, the fiber $\pi^{-1}(e_1,e_2)$ has codimension $n+2$. Now, inside $\text{Z}(a_{0,0,\ldots,0;1},\ldots,a_{0,0,\ldots,0;n})$, we have that $$a_{0,0,\ldots,0;1}=\cdots=a_{0,0,\ldots,0;k}=0,$$ which describes a subspace of codimension $n$, so the fiber has codimension $2n$. Therefore, since $n\geq 2$, we conclude that $\pi^{-1}(e_1,e_2)$ has codimension $n+2$. Now, for a generic fiber, we have $$\text{dim}(\pi_1^{-1}(p_1,p_2))=\text{dim}(X)-2n,$$ and since we can suppose that $\{p_1,p_2\}$ is linearly independent, we obtain $$\text{dim}(X)-\text{dim}(H(d_1,\ldots,d_n))=n-2.$$ Now, by Theorem \ref{tdf}, we obtain $$\text{dim}(\pi_2^{-1}(F))=\text{dim}(X)-\text{dim}(H(d_1,\ldots,d_n))=n-2.$$ For such $F$, we have that $\text{dim}(\pi_2^{-1}(F))=n-2$, so if we project this set on the first and second factors of $\mathbb{C}^n\times\mathbb{C}^n$, we get two sets of dimension less than or equal to $n-2$, and now we take their union, which is a set of dimension less than or equal to $n-2$ in $\mathbb{C}^n$. Outside of this set, if there existed $x,y\in\text{C}(F)$ such that $F(x)=F(y)$, we would have that $(x,y,F)\in\pi_2^{-1}(F)$, which contradicts the choices of $x$ and $y$. Therefore, $F$ is injective outside of this set of dimension less than or equal to $n-2$.
	
\end{proof}

\begin{lemma}\label{lemma4}
There exists a dense Zariski open set $U\subseteq H(d_1,\ldots,d_n)$ such that, for every $F\in U$, the following holds: If $p\in\Delta(F)$, then $|F^{-1}(p)\cap\text{C}(F)|\leq n-1$.
\end{lemma}

\begin{proof}
Define $$X=\left\{(p_1,\ldots,p_n,F)\in(\mathbb{C}^n)^n\times H_n(d_1,\ldots,d_n)\left| \begin{array}{c}F(p_1)=\cdots=F(p_n),\\ J(F)(p_i)=0,\text{ }1\leq i\leq n\end{array}\right.\right\}.$$ Once again, we can assume that $\{p_1,\ldots,p_n\}$ is linearly independent, so it suffices to determine the dimension of the fiber $\pi_1^{-1}(e_1,\ldots,e_n)$. The equation $$F(e_1)=F(e_2)=F(e_3)=\cdots=F(e_n)$$ yields \begin{equation}\label{exp2}a_{0,0,\ldots,0;k}=a_{d_k,0,\ldots,0;k}=a_{0,d_k,\ldots,0;k}=\cdots=a_{0,0,\ldots,d_k;k},\text{ for}\ 1\leq k\leq n,\end{equation} describing a subspace of codimension $n(n-1)$. 

Next, as we did in the previous lemma, the equations in \ref{exp2} make the $i$-th column of $J(F)(e_i)$ all the same, for $1\leq i\leq n$. If we expand each of those determinants with respect to the $i$-th column, we conclude that the equations $J(F)(e_i)=0$ still determine a subspace of  codimension $n$ inside the subspace of codimension $n(n-1)$ determined by \ref{exp2}. Therefore, $$\text{dim}(\pi_1^{-1}(e_1,\ldots,e_n))=\text{dim}(H(d_1,\ldots,d_n))-n^2.$$ Now, take a generic fiber, and we have $$\text{dim}(\pi_1^{-1}(p_1,\ldots,p_n))=\text{dim}(X)-n^2,$$ hence, $$\text{dim}(X)-\text{dim}(H(d_1,\ldots,d_n))=0.$$ Using Theorem \ref{tdf}, we conclude that $$\text{dim}(\pi_2^{-1}(F))=\text{dim}(X)-\text{dim}(H(d_1,\ldots,d_n))=0,$$ Therefore, $\text{dim}(\pi_2^{-1}(F))=0$ and, by homogeneity, we conclude that $|F^{-1}(p)\cap\text{C}(F)|\leq n-1$.

\end{proof}

\begin{remark}\label{rmk2}

In order to go further with our results and study the types of singular points that appear in a generic homogeneous polynomial mapping $F$, we need to describe such points through algebraic equations.

First, the singular points are characterized by the equation $J(F)=0$. Now, consider the following equations defined inductively: Define $J_{1,i}(F)$ as the determinant of the matrix obtained by replacing the $i$-th row of the Jacobian Matrix of $F$ by $\nabla J(F)$. Next, if $J_{k,i}(F)$ is defined, we define $J_{k+1,i}(F)$ as the determinant of the matrix obtained by replacing the $i$-th row of the Jacobian Matrix of $F$ by $\nabla J_{k,i}(F)$.

The above expressions, though being too large to write them explicitly, allow us to characterize some types of singularities under the $\mathcal{K}$-equivalence. Take $p\in C(F)$. Through a translation in the source, we can suppose that $p=0$. First, if $p$ is a singularity of corank greater than $1$, then all the minors of order $n-1$ of the Jacobian Matrix of $F$ at $p$ vanish. Thus, applying the Laplace expansion to each $J_{k,i}(F)$ with respect to the $i$-th row, it follows that $J_{k,i}(F)(p)=0$, for every $k$ and $i$. Therefore, those expressions do not help us to study singularities of corank greater than $1$.

Now, suppose that $p$ is a singularity of corank $1$. Then, a minor of order $n-1$ of the Jacobian Matrix of $F$ at $p$ does not vanish. Without loss of generality, suppose that $$\left|\begin{array}{ccc}\frac{\partial f_1}{\partial x_1}(p)&\cdots&\frac{\partial f_1}{\partial x_{n-1}}(p)\\\vdots&\ddots&\vdots\\\frac{\partial f_{n-1}}{\partial x_1}(p)&\cdots&\frac{\partial f_{n-1}}{\partial x_{n-1}}(p)\end{array}\right|\neq0.$$ Define $h:\mathbb{C}^n\to\mathbb{C}^n$ by $$h(x)=(f_1(x),\ldots,f_{n-1}(x),x_n).$$ If we expand its Jacobian determinant at $p$ with respect to the $n$-th row, we see that it is not zero. 
It follows from the Inverse Mapping Theorem that $h$ is a diffeomorphism germ at $p$. Observe that $$x=(h\circ h^{-1})(x)=(f_1(h^{-1}(x)),\ldots,f_{n-1}(h^{-1}(x)),f_n(h^{-1}(x))),$$ then $$(F\circ h^{-1})(x)=(f_1(h^{-1}(x)),\ldots,f_{n-1}(h^{-1}(x)),f_n(h^{-1}(x)))=(x_1,\ldots,x_{n-1},g(x)),$$ where $g(x)=f_n(h^{-1}(x))$. We conclude that $F$ is $\mathcal{A}$-equivalent and, in particular, $\mathcal{K}$-equivalent, to the much simpler germ $$G(x)=(F\circ h^{-1})(x)=(x_1,\ldots,x_{n-1},g(x)).$$ Furthermore, by applying the Weierstrass Preparation Theorem to $g$, it follows that, in a neighborhood of the origin, we can write $$g(x)=E(x_n)W(x',x_n),$$ where $x'=(x_1,\ldots,x_{n-1})$, $E$ does not vanish in this neighborhood and $$W(x',x_n)=x_n^{k+1}+x_n^kw_k(x')+\cdots+x_nw_1(x')+w_0(x'),$$ with $w_i(0)=0$, for $0\leq i\leq k$. This is the known Morin singularity $A_k$.

Even more, the equations $J_{k,i}(F)$ help us to determine the exponent $k$. Let us denote $F$ as $G$, since the exponent $k$ is $\mathcal{K}$-invariant. Then, the equations that we defined become much simpler. First, note that, since $p=0$ is a singular point of $F$, we have $k\geq1$. Now, since $p=0$ and by the expression of $F$, we have that $J_{r,i}(F)(p)=0$, for $r<k$ and $1\leq i\leq n-1$, since the last column vanishes at $p$. Now, considering the family $J_{r,n}(F)$, observe that $$J_{r,n}(F)=\left|\begin{array}{ccccc}1&0&\cdots&0&0\\0&1&\cdots&0&0\\\vdots&\vdots&\ddots&\vdots&\vdots\\0&0&\cdots&1&0\\\frac{\partial J_{r-1,n}(F)}{\partial x_1}&\frac{\partial J_{r-1,n}(F)}{\partial x_2}&\cdots&\frac{\partial J_{r-1,n}(F)}{\partial x_{n-1}}&\frac{\partial^{r+1}}{\partial x_n^{r+1}}(g(x))\end{array}\right|=\frac{\partial^{r+1}}{\partial x_n^{r+1}}(g(x)).$$ 
It follows that, whenever $r<k$, we also have $J_{r,n}(F)(p)=0$, and exactly for $r=k$, we get $J_{r,n}(F)(p)\neq0$. We conclude from this analysis that, if $k$ is the first index such that some $J_{k,i}(F)(p)$ is not zero, then $F$ has a singularity of type $A_k$ at $p$.
	
\end{remark}

Now we restrict our study to the $n=4$ case, because the following lemma does not hold in general:

\begin{lemma}\label{lemma5}
There exists a dense Zariski open set $U\subseteq H(d_1,d_2,d_3,d_4)$ such that, for every $F\in U$, all the singularities of $F$ outside the origin are only $A_1$ (fold), $A_2$ (cusp) and $A_3$ (swallowtail). In particular, the complement of the origin in the critical locus $\text{C}(F)$ is smooth.
\end{lemma}

\begin{proof}
	
Define $$X=\left\{(p,F)\in\mathbb{C}^4\times H(d_1, d_2, d_3, d_4)\left|\begin{array}{c}J(F)(p)=J_{1,i}(F)(p)=J_{2,i}(F)(p)=\\=J_{3,i}(F)(p)=0,\text{ }1\leq i\leq 4\end{array}\right.\right\}.$$ As we discussed in Remark \ref{rmk2}, if $(p,F)\in X$, then $p$ is either a singularity of corank greater than $1$ or of type $A_k$, with $k\geq4$. 

Let us compute the codimension of $\pi_1^{-1}(e_1)$. First, we expand the determinant $J(F)(e_1)$ with respect to the second column and we get that the equation $J(F)(e_1)=0$ can be written as $$0=\left|\begin{array}{cccc}d_1a_{0,0,0;1}&a_{1,0,0;1}&a_{0,1,0;1}&a_{0,0,1;1}\\d_2a_{0,0,0;2}&a_{1,0,0;2}&a_{0,1,0;2}&a_{0,0,1;2}\\d_3a_{0,0,0;3}&a_{1,0,0;3}&a_{0,1,0;3}&a_{0,0,1;3}\\d_4a_{0,0,0;4}&a_{1,0,0;4}&a_{0,1,0;4}&a_{0,0,1;4}\end{array}\right|=\sum_{k=1}^4(-1)^ka_{1,0,0;k}m_{k,1},$$ where $m_{k,1}$ is the minor obtained by removing the second column and the $k$-th row. Now, we shall compute the $4\times4$ minor $$\frac{\partial(J(F)(e_1),J_{1,1}(F)(e_1),J_{2,1}(F)(e_1),J_{3,1}(F)(e_1))}{\partial(a_{1,0,0;1},a_{2,0,0;1},a_{3,0,0;1},a_{4,0,0;1})}.$$ Observe that, from the expression $$f_1=\sum a_{i_2,i_3,i_4;1}x_1^{d_1-i_2-i_3-i_4}x_2^{i_2}x_3^{i_3}x_4^{i_4},$$ the variable $a_{i,0,0;1}$ appears in the above minor if, and only if, we take derivatives of $f_1$ with respect to the variable $x_2$ at least $i$ times. However, the determinant $J(F)$ involves only first-order derivatives, and $J_{k,1}$ involves derivatives of $f_1$ up to order $k+1$. Then, that minor has the form $$\left|\begin{array}{cccc}\frac{\partial J(F)(e_1)}{\partial a_{1,0,0;1}}&0&0&0\\\frac{\partial J_{1,1}(F)(e_1)}{\partial a_{1,0,0;1}}&\frac{\partial J_{1,1}(F)(e_1)}{\partial a_{2,0,0;1}}&0&0\\\frac{\partial J_{2,1}(F)(e_1)}{\partial a_{1,0,0;1}}&\frac{\partial J_{2,1}(F)(e_1)}{\partial a_{2,0,0;1}}&\frac{\partial J_{2,1}(F)(e_1)}{\partial a_{3,0,0;1}}&0\\\frac{\partial J_{3,1}(F)(e_1)}{\partial a_{1,0,0;1}}&\frac{\partial J_{3,1}(F)(e_1)}{\partial a_{2,0,0;1}}&\frac{\partial J_{3,1}(F)(e_1)}{\partial a_{3,0,0;1}}&\frac{\partial J_{3,1}(F)(e_1)}{\partial a_{4,0,0;1}}\end{array}\right|$$ Therefore, it is sufficient to compute the diagonal of that determinant. From the expanded expression of $J(F)(e_1)$, we see that $$\frac{\partial J(F)(e_1)}{\partial a_{1,0,0;1}}=-m_{1,1}.$$ Now, $J_{1,1}(F)$ consists of up to second-order derivatives, and the coefficient $a_{2,0,0;1}$ only appears when we take derivatives of $f_1$ with respect to $x_2$ at least twice. Thus, it only appears where we have the term$\frac{\partial^2f_1}{\partial x_2^2}(e_1)=2a_{2,0,0;1}$, and it is only possible to show up in the first row of every determinant that we write in the following computation. For this reason, we can use the formula for taking the derivative of a multilinear mapping, focusing on the term $\frac{\partial^2f_1}{\partial x_2^2}$, and obtain \begin{align*} \frac{\partial J_{1,1}(F)(e_1)}{\partial a_{2,0,0;1}}&=\frac{\partial}{\partial a_{2,0,0;1}}\left|\begin{array}{cccc}\frac{\partial J(F)}{\partial x_1}(e_1)&\frac{\partial J(F)}{\partial x_2}(e_1)&\frac{\partial J(F)}{\partial x_3}(e_1)&\frac{\partial J(F)}{\partial x_4}(e_1)\\\frac{\partial f_2}{\partial x_1}(e_1)&\frac{\partial f_2}{\partial x_2}(e_1)&\frac{\partial f_2}{\partial x_3}(e_1)&\frac{\partial f_2}{\partial x_4}(e_1)\\\frac{\partial f_3}{\partial x_1}(e_1)&\frac{\partial f_3}{\partial x_2}(e_1)&\frac{\partial f_3}{\partial x_3}(e_1)&\frac{\partial f_3}{\partial x_4}(e_1)\\\frac{\partial f_4}{\partial x_1}(e_1)&\frac{\partial f_4}{\partial x_2}(e_1)&\frac{\partial f_4}{\partial x_3}(e_1)&\frac{\partial f_4}{\partial x_4}(e_1)\end{array}\right|\\&=\left|\begin{array}{cccc}0&\frac{\partial}{\partial a_{2,0,0;1}}\left(\frac{\partial J(F)}{\partial x_2}(e_1)\right)&0&0\\\frac{\partial f_2}{\partial x_1}(e_1)&\frac{\partial f_2}{\partial x_2}(e_1)&\frac{\partial f_2}{\partial x_3}(e_1)&\frac{\partial f_2}{\partial x_4}(e_1)\\\frac{\partial f_3}{\partial x_1}(e_1)&\frac{\partial f_3}{\partial x_2}(e_1)&\frac{\partial f_3}{\partial x_3}(e_1)&\frac{\partial f_3}{\partial x_4}(e_1)\\\frac{\partial f_4}{\partial x_1}(e_1)&\frac{\partial f_4}{\partial x_2}(e_1)&\frac{\partial f_4}{\partial x_3}(e_1)&\frac{\partial f_4}{\partial x_4}(e_1)\end{array}\right|\\&=-m_{1,1}\frac{\partial}{\partial a_{2,0,0;1}}\left(\frac{\partial}{\partial x_2}\left|\begin{array}{cccc}\frac{\partial f_1}{\partial x_1}&\frac{\partial f_1}{\partial x_2}&\frac{\partial f_1}{\partial x_3}&\frac{\partial f_1}{\partial x_4}\\\frac{\partial f_2}{\partial x_1}&\frac{\partial f_2}{\partial x_2}&\frac{\partial f_2}{\partial x_3}&\frac{\partial f_2}{\partial x_4}\\\frac{\partial f_3}{\partial x_1}&\frac{\partial f_3}{\partial x_2}&\frac{\partial f_3}{\partial x_3}&\frac{\partial f_3}{\partial x_4}\\\frac{\partial f_4}{\partial x_1}&\frac{\partial f_4}{\partial x_2}&\frac{\partial f_4}{\partial x_3}&\frac{\partial f_4}{\partial x_4}\end{array}\right|(e_1)\right)\\&=-m_{1,1}\left|\begin{array}{cccc}0&\frac{\partial}{\partial a_{2,0,0;1}}\left(\frac{\partial^2 f_1}{\partial x_2^2}(e_1)\right)&0&0\\\frac{\partial f_2}{\partial x_1}(e_1)&\frac{\partial f_2}{\partial x_2}(e_1)&\frac{\partial f_2}{\partial x_3}(e_1)&\frac{\partial f_2}{\partial x_4}(e_1)\\\frac{\partial f_3}{\partial x_1}(e_1)&\frac{\partial f_3}{\partial x_2}(e_1)&\frac{\partial f_3}{\partial x_3}(e_1)&\frac{\partial f_3}{\partial x_4}(e_1)\\\frac{\partial f_4}{\partial x_1}(e_1)&\frac{\partial f_4}{\partial x_2}(e_1)&\frac{\partial f_4}{\partial x_3}(e_1)&\frac{\partial f_4}{\partial x_4}(e_1)\end{array}\right|	\\&=-m_{1,1}\left|\begin{array}{cccc}0&2&0&0\\\frac{\partial f_2}{\partial x_1}(e_1)&\frac{\partial f_2}{\partial x_2}(e_1)&\frac{\partial f_2}{\partial x_3}(e_1)&\frac{\partial f_2}{\partial x_4}(e_1)\\\frac{\partial f_3}{\partial x_1}(e_1)&\frac{\partial f_3}{\partial x_2}(e_1)&\frac{\partial f_3}{\partial x_3}(e_1)&\frac{\partial f_3}{\partial x_4}(e_1)\\\frac{\partial f_4}{\partial x_1}(e_1)&\frac{\partial f_4}{\partial x_2}(e_1)&\frac{\partial f_4}{\partial x_3}(e_1)&\frac{\partial f_4}{\partial x_4}(e_1)\end{array}\right|\\&=2m_{1,1}^2.\end{align*} With the same argument, we obtain $$\frac{\partial J_{i-1,1}(F)(e_1)}{\partial a_{i,0,0;1}}=(-1)^ii!m_{1,1}^i.$$
Therefore, \begin{align*}\frac{\partial(J(F)(e_1),J_{1,1}(F)(e_1),J_{2,1}(F)(e_1),J_{3,1}(F)(e_1))}{\partial(a_{1,0,0;1},a_{2,0,0;1},a_{3,0,0;1},a_{4,0,0;1})}&=\prod_{i=1}^4(-1)^ii!m_{1,1}^i\\&=288m_{1,1}^{10}\end{align*} We see that outside of the set $\text{Z}(m_{1,1})$, the above $4\times 4$ minor is not zero, thus, those $4$ hypersurfaces intersect transversally, which implies that $$\text{dim}(\pi_1^{-1}(e_1))= \text{dim}(H(d_1,d_2,d_3,d_4))-4$$ outside $\text{Z}(m_{1,1})$. The same conclusion is true outside of $\text{Z}(m_{k,1})$, for $1\leq k\leq 4$, computing $$\frac{\partial(J(F)(e_1),J_{1,k}(e_1),J_{2,k}(e_1),J_{3,k}(F)(e_1))}{\partial(a_{1,0,0;k},a_{2,0,0;k},a_{3,0,0;k},a_{4,0,0;k})}=288m_{k,1}^{10}.$$ Now, inside $\text{Z}(m_{1,1},m_{2,1},m_{3,1},m_{4,1})$, we have the vanishing of all minors of order $3\times3$ of the $4\times3$ matrix $$\left[\begin{array}{ccc}d_1a_{0,0,0;1}&a_{0,1,0;1}&a_{0,0,1;1}\\d_2a_{0,0,0;2}&a_{0,1,0;2}&a_{0,0,1;2}\\d_3a_{0,0,0;3}&a_{0,1,0;3}&a_{0,0,1;3}\\d_4a_{0,0,0;4}&a_{0,1,0;4}&a_{0,0,1;4}\end{array}\right],$$ and this describes a determinantal variety that has codimension $2$. In order to lose two additional dimensions, we make analogous computations, expanding the determinant $J(F)(e_1)$ with respect to the third and fourth columns, obtaining $$\sum_{k=1}^4(-1)^{k+1}a_{0,1,0;k}m_{k,2}=\sum_{k=1}^4(-1)^ka_{0,0,1;k}m_{k,3}=0,$$ 
where $m_{k,2}$ and $m_{k,3}$ are the corresponding $3\times3$ minors. As we did above, outside each $\text{Z}(m_{k,2})$ and $\text{Z}(m_{k,3})$, we have a transversal intersection of $4$ hypersurfaces, which gives a subset of codimension $4$. Finally, we see that the set $$\text{Z}(m_{1,1},m_{2,1},m_{3,1},m_{4,1},m_{1,2},m_{2,2},m_{3,2},m_{4,2},m_{1,3},m_{2,3},m_{3,3},m_{4,3})$$ has codimension $4$.
	
	Therefore, we have that the fiber $\pi_1^{-1}(e_1)$ has codimension $4$. Again, applying Theorem \ref{tdf}, we can take $p\in\mathbb{C}^4$ sufficiently generic and obtain $$\text{dim}(H(d_1,d_2,d_3,d_4))-4=\text{dim}(\pi_1^{-1}(p))=\text{dim}(X)-4,$$ that is, $$\text{dim}(X)-\text{dim}(H(d_1,d_2,d_3,d_4))=0.$$ On the other hand, by Theorem \ref{tdf}, for $F\in H(d_1,d_2,d_3,d_4)$ generic, we have $$\text{dim}(\pi_2^{-1}(F))=\text{dim}(X)-\text{dim}(H(d_1,d_2,d_3,d_4))=0.$$ Since $X$ is defined by homogeneous equations, we conclude that $\pi_2^{-1}(F)=\{(0,F)\}$. Therefore, for such $F$, if $F$ had a singularity that is neither $A_1$, $A_2$ nor $A_3$ at a point $p\in\mathbb{C}^4$ that is not the origin, we would have $(p,F)\in\pi_2^{-1}(F)$, which is a contradiction.

\end{proof}

\begin{remark}
	
The above proof cannot be generalized for higher dimensions, that is, for $F\in H(d_1,d_2,\ldots,d_n)$, with $n\geq 5$. Though we see that all the computations are analogous and we obtain \begin{align*}\frac{\partial(J(F)(e_1),J_{1,k}(F)(e_1),\ldots,J_{n-1,k}(F)(e_1))}{\partial(a_{1,0,\ldots,0;k},a_{2,0,\ldots,0;k},\ldots,a_{n,0,\ldots,0;k})}&=\prod_{i=1}^n(-1)^ii!m_{k,1}^i\\&=\left(\prod_{i=1}^n(-1)^ii!\right)m_{k,1}^\frac{n(n+1)}{2},\end{align*} which lets us conclude that, outside of each $\text{Z}(m_{k,1})$, we have a transversal intersection of $n$ hypersurfaces, where we lose the $n$ dimensions that we want, and the same conclusion is true with the other minors coming from the expansion of $J(F)(e_1)$ with respect to the other columns, we cannot finish the proof, because the subset $$\text{Z}(m_{1,1},\ldots,m_{n,1},m_{1,2},\ldots,m_{n,2},\ldots,m_{1,n-1},\ldots,m_{n,n-1})$$ still has codimension $4$, which prohibits us from losing the $n$ dimensions that is needed.

This is expected to happen, because corank $2$ singularities define a subset of codimension $4$. It implies that for a generic mapping $F:\mathbb{C}^5\to\mathbb{C}^5$, those points will appear along curves, so it cannot happen only at the origin.

For this reason, from now on, we focus our results to the $n=4$ case.
	
\end{remark}

\begin{lemma}\label{lemma6}
	
There exists a dense Zariski open set $U\subseteq H(d_1,d_2,d_3,d_4)$ such that, for every $F\in U$, if $F$ has a swallowtail at $p$, then $F^{-1}(F(p))\cap \text{C}(F)=\{p\}$.

\end{lemma}

\begin{proof}
	
	Let us $$X=\left\{(p,q,F)\in(\mathbb{C}^4)^2\times H(d_1,d_2,d_3,d_4)\left|\begin{array}{c}F(p)=F(q),\text{ }J(F)(p)=J(F)(q)=0,\\ J_{1,i}(F)(p)=J_{2,i}(F)(p)=0,\text{ }1\leq i\leq 4 \end{array}\right.\right\}.$$ As before, it suffices to compute the dimension of the fiber $\pi_1^{-1}(e_1,e_2)$. The equation $F(e_1)=F(e_2)$ is $$a_{0,0,0;k}=a_{d_k,0,0;k},\text{ for }1\leq k\leq 4,$$ which defines a subspace of codimension $4$ in $H(d_1,d_2,d_3,d_4)$. Furthermore, as we did in the previous lemma, the equations $$J(F)(e_1)=J_{1,i}(F)(e_1)=J_{2,i}(F)(e_1)=0,\text{ for }1\leq i\leq 4,$$ even when restricted to the subspace determined by the four equations above, also define a subspace of codimension $3$ outside $$\text{Z}(m_{1,1},m_{2,1},m_{3,1},m_{4,1}),$$ and the same is true outside $$\text{Z}(m_{1,2},m_{2,2},m_{3,2},m_{4,2}),$$ thus, all these equations give a subspace of codimension $7$ outside those sets. Now, the subset $$\text{Z}(m_{1,1},m_{2,1},m_{3,1},m_{4,1},m_{1,2},m_{2,2},m_{3,2},m_{4,2})$$ has codimension $3$, which totals codimension $7$. Finally, the equation $J(F)(e_2)=0$ is $$\left|\begin{array}{cccc}a_{d_1-1,0,0;1}&d_1a_{d_1,0,0;1}&a_{d_1-1,1,0;1}&a_{d_1-1,0,1;1}\\a_{d_2-1,0,0;2}&d_2a_{d_2,0,0;2}&a_{d_2-1,1,0;2}&a_{d_2-1,0,1;2}\\a_{d_3-1,0,0;3}&d_3a_{d_3,0,0;3}&a_{d_3-1,1,0;3}&a_{d_3-1,0,1;3}\\a_{d_4-1,0,0;4}&d_4a_{d_4,0,0;4}&a_{d_4-1,1,0;4}&a_{d_4-1,0,1;4}\end{array}\right|=0,$$ which is an independent equation from the previous ones, since it involves a set of variables distinct than the variables involved in the previous equations, and it gives codimension $8$. Therefore, $$\text{dim}(\pi_1^{-1}(e_1,e_2))=\text{dim}(H(d_1,d_2,d_3,d_4))-8.$$ Now, from Theorem \ref{tdf}, we have that a generic fiber of $\pi_1$ has dimension $$\text{dim}(\pi_1^{-1}(p,q))=\text{dim}(H(d_1,d_2,d_3,d_4))-8,$$ so $$\text{dim}(X)-\text{dim}(H(d_1,d_2,d_3,d_4))=0.$$ Therefore, applying Theorem \ref{tdf}, we obtain $$\text{dim}(\pi_2^{-1}(F))=\text{dim}(X)-\text{dim}(H(d_1,d_2,d_3,d_4))=0,$$ for $F\in H(d_1,d_2,d_3,d_4)$ generic. For such $F$, if $F$ has a swallowtail at $p$, then we already have $$\{p\}\subseteq F^{-1}(F(p))\cap\text{C}(F).$$ For the other inclusion, if we had a point $q\in (F^{-1}(F(p))\cap\text{C}(F))\setminus\{p\}$, then the triple $(p,q,F)$ is an element of the fiber $\pi_2^{-1}(F)$, thus $(\lambda p,\lambda q,F)\in\pi_2^{-1}(F)$, for every $\lambda\in\mathbb{C}$, which is a contradiction with the dimension of this fiber.

\end{proof}

\begin{remark}
	
	For $F\in H(d_1,\ldots,d_n)$ generic, we have that $\Delta(F)$ is a hypersurface. In order to see this, consider the restriction $F:\text{C}(F)\to\overline{\Delta(F)}$, which is a dominant morphism. We can apply Theorem \ref{tdf} and take a generic point $y\in\overline{\Delta(F)}$ such that $$\text{dim}(F^{-1}(y))=\text{dim}(\text{C}(F))-\text{dim}(\Delta(F))=n-1-\text{dim}(\overline{\Delta(F)}).$$ From Remark \ref{rmk4}, we know that a generic $F\in H(d_1,\ldots,d_n)$ is proper, then $F^{-1}(y)$ is a compact algebraic set. Applying \cite[Geometric form of Noether's normalization lemma, p. 42]{3} to each of its irreducible component, we see that $F^{-1}(y)$ is a set of finite points, thus $$0=\text{dim}(F^{-1}(y))=n-1-\text{dim}(\overline{\Delta(F)}).$$ Therefore, $\text{dim}(\overline{\Delta(F)})=n-1,$ that is, $\overline{\Delta(F)}$ and $\Delta(F)$ are hypersurfaces.
	
\end{remark}

\begin{remark}\label{rmk3}
	
Given two hypersurfaces $M$ and $N$ in $\mathbb{C}^n$, and a point $p \in M \cap N$, then either $M \pitchfork N$ at $p$, or $T_pM = T_pN$. Indeed, if $M$ does not intersect $N$ transversally at $p$, then we have that
$$T_pM \subseteq T_pM + T_pN \subsetneq \mathbb{C}^n,$$
which implies
$$n-1 \leq \text{dim}(T_pM + T_pN) < n \Rightarrow \text{dim}(T_pM + T_pN) = n-1.$$ Since $T_pM, T_pN \subseteq T_pM + T_pN$ with the same dimensions, we conclude that
$$T_pM = T_pM + T_pN = T_pN,$$ as we wanted.
	
\end{remark}

\begin{lemma}\label{lemma7}
	There exists a dense Zariski open set $U \subseteq H(d_1,d_2,d_3,d_4)$ such that, for every $F \in U$, if $|F^{-1}(p) \cap \text{C}(F)| \geq 2$, then the hypersurface $\Delta(F)$ has a normal crossing at $p$.
\end{lemma}

\begin{proof}
	
With the above remarks, since $\Delta(F)$ is a hypersurface, given two components of $\Delta(F)$ and a point in their intersection, either they intersect transversally or they have equal tangent spaces at that point. Thus, let us define $$X=\{(p,q,F)\in(\mathbb{C}^4)^2\times H_4(d_1,d_2,d_3,d_4)\mid F(p)=F(q),\text{ }d_pF(\mathbb{C}^4)=d_qF(\mathbb{C}^4)\}$$ and it suffices to obtain the dimension of the fiber $\pi_1^{-1}(e_1,e_2)$. Again, the equation $F(e_1)=F(e_2)$ is $$a_{0,0,0;k}=a_{d_k,0,0;k},\text{ for }1\leq k\leq 4,$$ which define a subspace of codimension $4$. Now, we have that $$d_{e_1}F=\left[\begin{array}{cccc}d_1a_{0,0,0;1}&a_{1,0,0;1}&a_{0,1,0;1}&a_{0,0,1;1}\\d_2a_{0,0,0;2}&a_{1,0,0;2}&a_{0,1,0;2}&a_{0,0,1;2}\\d_3a_{0,0,0;3}&a_{1,0,0;3}&a_{0,1,0;3}&a_{0,0,1;3}\\d_4a_{0,0,0;4}&a_{1,0,0;4}&a_{0,1,0;4}&a_{0,0,1;4}\end{array}\right]$$ and $$d_{e_2}F=\left[\begin{array}{cccc}a_{d_1-1,0,0;1}&d_1a_{d_1,0,0;1}&a_{d_1-1,1,0;1}& a_{d_1-1,0,1;1}\\a_{d_2-1,0,0;2}&d_2a_{d_2,0,0;2}&a_{d_2-1,1,0;2} &a_{d_2-1,0,1;2}\\a_{d_3-1,0,0;3}&d_3a_{d_3,0,0;3}&a_{d_3-1,1,0;3} &a_{d_3-1,0,1;3}\\a_{d_4-1,0,0;4}&d_4a_{d_4,0,0;4}&a_{d_4-1,1,0;4} &a_{d_4-1,0,1;4}\end{array}\right].$$ Since the vector space $d_{e_i}F(\mathbb{C}^4)$ is generated by the columns of the matrix $d_{e_i}F$, then the equality  $d_{e_1}F(\mathbb{C}^4)=d_{e_2}F(\mathbb{C}^4)$ holds if, and only if, the columns of the second matrix are linearly dependent with the columns of the first matrix, that is, the matrix obtained by the concatenation of these two matrices, $$A=\left[\begin{array}{cccccccc}d_1a_{0,0,0;1}&a_{1,0,0;1}&a_{0,1,0;1}&a_{0,0,1;1}&a_{d_1-1,0,0;1}&d_1a_{d_1,0,0;1}&a_{d_1-1,1,0;1}&a_{d_1-1,0,1;1}\\d_2a_{0,0,0;2}&a_{1,0,0;2}&a_{0,1,0;2}&a_{0,0,1;2}&a_{d_2-1,0,0;2}&d_2a_{d_2,0,0;2}&a_{d_2-1,1,0;2}&a_{d_2-1,0,1;2}\\d_3a_{0,0,0;3}&a_{1,0,0;3}&a_{0,1,0;3}&a_{0,0,1;3}&a_{d_3-1,0,0;3}&d_3a_{d_3,0,0;3}&a_{d_3-1,1,0;3}&a_{d_3-1,0,1;3}\\d_4a_{0,0,0;4}&a_{1,0,0;4}&a_{0,1,0;4}&a_{0,0,1;4}&a_{d_4-1,0,0;4}&d_4a_{d_4,0,0;4}&a_{d_4-1,1,0;4}&a_{d_4-1,0,1;4}\end{array}\right],$$ must have rank less than $4$, in other words, all the $4\times4$ minors vanish. Note that, inside the subspace defined by the equation $F(e_1)=F(e_2)$, the first and the sixth columns are equal, so we can remove the sixth column without changing the matrix rank. Thus, the non-transversality condition becomes equivalent to the vanishing of all $4\times4$ minors of the matrix $$B=\left[\begin{array}{ccccccc}d_1a_{0,0,0;1}&a_{1,0,0;1}&a_{0,1,0;1}&a_{0,0,1;1}&a_{d_1-1,0,0;1}&a_{d_1-1,1,0;1}&a_{d_1-1,0,1;1}\\d_2a_{0,0,0;2}&a_{1,0,0;2}&a_{0,1,0;2}&a_{0,0,1;2}&a_{d_2-1,0,0;2}&a_{d_2-1,1,0;2}&a_{d_2-1,0,1;2}\\d_3a_{0,0,0;3}&a_{1,0,0;3}&a_{0,1,0;3}&a_{0,0,1;3}&a_{d_3-1,0,0;3}&a_{d_3-1,1,0;3}&a_{d_3-1,0,1;3}\\d_4a_{0,0,0;4}&a_{1,0,0;4}&a_{0,1,0;4}&a_{0,0,1;4}&a_{d_4-1,0,0;4}&a_{d_4-1,1,0;4}&a_{d_4-1,0,1;4}\end{array}\right].$$ This describes a determinantal variety of codimension $4$, and this fact can be seen as follows: For each $3\times3$ minor that does not vanish, there exists four $4\times 4$ minors that contain it. For instance, suppose that $$\left|\begin{array}{ccc}d_1a_{0,0,0;1}&a_{1,0,0;1}&a_{0,1,0;1}\\d_2a_{0,0,0;2}&a_{1,0,0;2}&a_{0,1,0;2}\\d_3a_{0,0,0;3}&a_{1,0,0;3}&a_{0,1,0;3}\end{array}\right|\neq0.$$ Then, the four $4\times 4$ minors that contain it are formed by adding to it one of the four columns and a portion of the last line, and we expand this determinant with respect to the added column, as follows: \begin{align*}g_1=\left|\begin{array}{cccc}d_1a_{0,0,0;1}&a_{1,0,0;1}&a_{0,1,0;1}&a_{0,0,1;1}\\d_2a_{0,0,0;2}&a_{1,0,0;2}&a_{0,1,0;2}&a_{0,0,1;2}\\d_3a_{0,0,0;3}&a_{1,0,0;3}&a_{0,1,0;3}&a_{0,0,1;3}\\d_4a_{0,0,0;4}&a_{1,0,0;4}&a_{0,1,0;4}&a_{0,0,1;4}\end{array}\right|&=\sum_{k=1}^4(-1)^{k}a_{0,0,1;k}m_k\\g_2=\left|\begin{array}{cccc}d_1a_{0,0,0;1}&a_{1,0,0;1}&a_{0,1,0;1}&a_{d_1-1,0,0;1}\\d_2a_{0,0,0;2}&a_{1,0,0;2}&a_{0,1,0;2}&a_{d_2-1,0,0;2}\\d_3a_{0,0,0;3}&a_{1,0,0;3}&a_{0,1,0;3}&a_{d_3-1,0,0;3}\\d_4a_{0,0,0;4}&a_{1,0,0;4}&a_{0,1,0;4}&a_{d_4-1,0,0;4}\end{array}\right|&=\sum_{k=1}^4(-1)^{k}a_{d_k-1,0,0;k}m_k\\g_3=\left|\begin{array}{cccc}d_1a_{0,0,0;1}&a_{1,0,0;1}&a_{0,1,0;1}&a_{d_1-1,1,0;1}\\d_2a_{0,0,0;2}&a_{1,0,0;2}&a_{0,1,0;2}&a_{d_2-1,1,0;2}\\d_3a_{0,0,0;3}&a_{1,0,0;3}&a_{0,1,0;3}&a_{d_3-1,1,0;3}\\d_4a_{0,0,0;4}&a_{1,0,0;4}&a_{0,1,0;4}&a_{d_4-1,1,0;4}\end{array}\right|&=\sum_{k=1}^4(-1)^{k}a_{d_k-1,1,0;k}m_k\\g_4=\left|\begin{array}{cccc}d_1a_{0,0,0;1}&a_{1,0,0;1}&a_{0,1,0;1}&a_{d_1-1,0,1;1}\\d_2a_{0,0,0;2}&a_{1,0,0;2}&a_{0,1,0;2}&a_{d_2-1,0,1;2}\\d_3a_{0,0,0;3}&a_{1,0,0;3}&a_{0,1,0;3}&a_{d_3-1,0,1;3}\\d_4a_{0,0,0;4}&a_{1,0,0;4}&a_{0,1,0;4}&a_{d_4-1,0,1;4}\end{array}\right|&=\sum_{k=1}^4(-1)^{k}a_{d_k-1,0,1;k}m_k,\end{align*} where $m_1$, $m_2$, $m_3$ and $m_4$ are the respective minors, and the minor that we are supposing that does not vanish is $m_4$. Note that the variable that is multiplying it is different for each $g_i$. Thus, those four equations are independent, since $$\frac{\partial(g_1,g_2,g_3,g_4)}{\partial(a_{0,0,1;4},a_{d_4-1,0,0;4},a_{d_4-1,1,0;4},a_{d_4-1,0,1;4})}=\left|\begin{array}{cccc}m_4 & 0 & 0 & 0\\0 & m_4 & 0 & 0\\ 0 & 0 & m_4 & 0\\0 & 0 & 0 & m_4\end{array}\right|=m_4^4\neq0.$$ Therefore, those four hypersurfaces intersect transversally. We conclude that, inside the subspace defined by the equation $F(e_1)=F(e_2)$, the fiber $\pi_1^{-1}(e_1,e_2)$ has codimension $4$ outside the set where the $3\times 3$ minors vanish. Now, inside this set, we can repeat the same argument with the vanishing of $2\times 2$ minors. If any of such minors does not vanish, there exists more than four $3\times3$ minors that contain it, and it follows that the fiber has codimension at least $4$. Finally, the set defined by the vanishing of all $2\times2$ minors has codimension at least $4$, since $B$ has more than four $2\times2$ minors that involves distinct variables between themselves.
	
Therefore, $F(e_1)=F(e_2)$ defines a subspace of codimension $4$, and the condition $d_{e_1}F(\mathbb{C}^4)=d_{e_2}F(\mathbb{C}^4)$, when restricted to that subspace, also defines a subspace of codimension $4$, so we conclude that the fiber $\pi_1^{-1}(e_1,e_2)$ has codimension $8$, that is, $$\text{dim}(\pi_1^{-1}(e_1,e_2))= \text{dim}(H_4(d_1,d_2,d_3,d_4))-8.$$ Now, we can take a generic fiber and obtain $$\text{dim}(X)-\text{dim}(H_4(d_1,d_2,d_3,d_4))=0.$$ Applying Theorem \ref{tdf} one last time to $\pi_2$, we get $$\text{dim}(\pi_2^{-1}(F))=\text{dim}(X)-\text{dim}(H_4(d_1,d_2,d_3,d_4))=0,$$ for generic $F\in H_4(d_1,d_2,d_3,d_4)$. Therefore, $\text{dim}(\pi_2^{-1}(F))=0$ and, by homogeneity, we obtain $$\pi_2^{-1}(F)=\{(0,0,F)\}.$$ For such $F$, if we suppose that $|F^{-1}(p)\cap\text{C}(F)|\geq 2$, then we take $q\in F^{-1}(p)\cap\text{C}(F)$ that is different than $p$ and we obtain $F(p)=F(q)$, but $(p,q,F)\notin\pi_2^{-1}(F)$, so, in order to not get a contradiction, it follows that $d_pF(\mathbb{C}^4)\neq d_qF(\mathbb{C}^4)$, that is, the transversality condition holds. Therefore, $\Delta(F)$ has a normal crossing at $p$.

\end{proof}

Now we can proceed exactly as in \cite{1}. For the sake of completeness, we present it here as well. We apply the following characterization of $\mathcal{A}$-finite determinancy, and a proof can be found in \cite{5}.

\begin{theorem}[Geometric Criterion for $\mathcal{A}$-finite determinancy]\label{teo2}

Let $F:(\mathbb{C}^n,0)\to(\mathbb{C}^n,0)$ be a mapping germ. Then $F$ is $\mathcal{A}$-finitely determined if, and only if, there exists a representative $F:U\to V$ such that $F^{-1}(0)=\{0\}$ and $F\big|_{U\setminus\{0\}}:U\setminus\{0\}\to V\setminus\{0\}$ is stable.

\end{theorem}

We also make use of the following result regarding stability, and a proof can be found in \cite[Theorem 1, p.267]{6}.

\begin{theorem}[Infinitesimal Stability Implies Stability]\label{teo3}

Let $N$ and $P$ be smooth manifolds (not necessarily compact). If $f:N\to P$ is a proper mapping and infinitesimally stable, then $f$ is stable.

\end{theorem}

With these results, we can prove another main theorem with a partial converse about the relationship between the genericity of $\mathcal{A}$-finite determinancy and the degrees of homogeneous polynomial mappings.

\begin{theorem}\label{teo4}
If $\text{gcd}(d_i,d_j)\leq 2$, for $1\leq i<j\leq 4$, and $\text{gcd}(d_i,d_j,d_k)=1$, for $1\leq i<j<k\leq 4$, then there exists a dense Zariski open set $U\subseteq H(d_1,d_2,d_3,d_4)$ such that, for every $F\in U$, we have that the germ $F:(\mathbb{C}^4,0)\to(\mathbb{C}^4,0)$ is $\mathcal{A}$-finitely determined.

Moreover, if $\text{gcd}(d_1,d_2,d_3,d_4)>1$, then there are no homogeneous polynomial mapping germ of degrees $d_1,d_2,d_3,d_4$ that is $\mathcal{A}$-finitely determined.
\end{theorem}

\begin{proof}
Let $U$ be the dense Zariski open set obtained in Theorem \ref{teo1}. Since $F$ is homogeneous and $F^{-1}(0)=\{0\}$, we have from Proposition \ref{prop1} that $F$ is proper. Now, using the properties of $F$ as in Theorem \ref{teo1}, we conclude that $F$ is infinitesimally stable outside the origin by applying \cite[p. 138, Corollary 4.8]{8}, because of the singularities that it has outside the origin, $F$ is a two-to-one mapping and it satisfies the normal crossing condition. Applying Theorem \ref{teo3}, we conclude that $F$ is stable outside the origin. Therefore, by Theorem \ref{teo2}, we can conclude that the germ $F:(\mathbb{C}^4,0)\to(\mathbb{C}^4,0)$ is $\mathcal{A}$-finitely determined.

Moreover, if $\text{gcd}(d_1,d_2,d_3,d_4)>1$, suppose that there exists a $\mathcal{A}$-finitely determined germ $F:(\mathbb{C}^4,0)\to(\mathbb{C}^4,0)$. Then, as proved in Remark \ref{rmk1}, we can see that $F$ is not injective in any ray contained in $\text{C}(F)$. In \cite[p.534, Proposition D.4]{8} we see that the restriction $F\big|_{\text{C}(F)}:\text{C}(F)\to\Delta(F)$ is a normalization, and one of the conditions is that this restriction is generically one-to-one, which is a contradiction. Consequently, the germ cannot be $\mathcal{A}$-finitely determined.
\end{proof}

\subsection{Counting singularities}

Following the work from the authors in \cite{1}, we can make use of Ohmoto's results in \cite{7} and Kazarian's results in \cite{9} to obtain enumerative formulas of the discrete mono and multi-singularities of generic complex polynomial mappings from $\mathbb{C}^4$ to $\mathbb{C}^4$ in terms of its degrees $d_i$.

Let $F=(f_1,f_2,f_3,f_4)\in\Omega(d_1,d_2,d_3,d_4)$ be a complex polynomial mapping, $\overline{f_i}$ be the homogeneous part of degree $d_i$, and consider $F_0=(\overline{f_1},\overline{f_2},\overline{f_3},\overline{f_4})\in H(d_1,d_2,d_3,d_4)$. For a generic choice of the coefficients of those homogeneous parts, we can assume from Theorem \ref{teo4} that the germ $(F_0,0)$ is $\mathcal{A}$-finitely determined. In this case, consider the deformation $$F_t(x)=(t^{d_1}f_1(x/t),t^{d_2}f_2(x/t),t^{d_3}f_3(x/t),t^{d_4}f_4(x/t)).$$ We see that all of its singularities goes to the origin when we make $t\to0$, so we can apply Theorem \cite[Theorem 5.3]{7}. Consider the toric actions of $\mathbb{C}^\ast$ in $\mathbb{C}^4$ given by $$\rho_1=\alpha\oplus\alpha\oplus\alpha\oplus\alpha\text{ and }\rho_2=\alpha^{d_1}\oplus\alpha^{d_2}\oplus\alpha^{d_3}\oplus\alpha^{d_4},$$ with the respective fiber bundles $$E_1=\ell\oplus\ell\oplus\ell\oplus\ell\text{ and }E_2=\ell^{\otimes d_1}\oplus\ell^{\otimes d_2}\oplus\ell^{\otimes d_3}\oplus\ell^{\otimes d_4},$$ where $\ell$ is the tautological line bundle over some $\mathbb{P}_\mathbb{C}^N$, with $N$ sufficiently large, and we construct the universal mapping $f_0:E_1\to E_2$ such that its restriction to each fiber is $\mathcal{A}$-equivalent to $f_0:\mathbb{C}^4\to\mathbb{C}^4$. Let us denote by $a=c_1(\ell)$ and we compute the following Chern classes \begin{align*}
	c(f_0)&=c(f_0^\ast E_2-E_1)\\&=\frac{c(E_2)}{c(E_1)}\\&=\frac{(1+d_1a)(1+d_2a)(1+d_3a)(1+d_4a)}{(1+a)^4}\\&=\frac{(1+d_1a)(1+d_2a)(1+d_3a)(1+d_4a)}{1 -(-4a-6a^2-4a^3-a^4)}\\&=(1+d_1a)(1+d_2a)(1+d_3a)(1+d_4a)\left(\sum_{i=0}^\infty(-4a-6a^2-4a^3-a^4)^i\right)\\&=1 + \left(d_{1} + d_{2} + d_{3} + d_{4} - 4\right) a\\& + \left(d_{1} d_{2} + d_{1} d_{3} + d_{1} d_{4} + d_{2} d_{3} + d_{2} d_{4} + d_{3} d_{4} - 4 d_{1} - 4 d_{2} - 4 d_{3} - 4 d_{4} + 10\right) a^{2} \\&+ \left(d_{1} d_{2} d_{3} + d_{1} d_{2} d_{4} + d_{1} d_{3} d_{4} + d_{2} d_{3} d_{4} - 4 d_{1} d_{2} - 4 d_{1} d_{3} - 4 d_{1} d_{4} - 4 d_{2} d_{3} \right.\\&\left.- 4 d_{2} d_{4} - 4 d_{3} d_{4} + 10 d_{1} + 10 d_{2} + 10 d_{3} + 10 d_{4} - 20\right) a^{3} \\&+ \left(d_{1} d_{2} d_{3} d_{4} - 4 d_{1} d_{2} d_{3} - 4 d_{1} d_{2} d_{4} - 4 d_{1} d_{3} d_{4} - 4 d_{2} d_{3} d_{4} + 10 d_{1} d_{2} + 10 d_{1} d_{3}\right.\\&\left.+ 10 d_{1} d_{4} + 10 d_{2} d_{3} + 10 d_{2} d_{4} + 10 d_{3} d_{4} - 20 d_{1} - 20 d_{2} - 20 d_{3} - 20 d_{4} + 35\right) a^{4}+\cdots.
\end{align*} Furthermore, since $f_0$ maps the zero section of $E_1$ to the zero section of $E_2$, then $$c_4(E_2)=f_{0\ast}(c_4(E_1))=c_4(E_1)f_{0\ast}(1),$$ so, $$f_{0\ast}(1)=\frac{c_4(E_2)}{c_4(E_1)}=\frac{d_1d_2d_3d_4a^4}{a^4}=d_1d_2d_3d_4.$$ Thus, we obtained the Chern and Landweber-Novikov classes. To simplify the enumerative formulas, we denote by\begin{align*}
	c_1&=d_{1} + d_{2} + d_{3} + d_{4} - 4\\
	c_2&=d_{1} d_{2} + d_{1} d_{3} + d_{1} d_{4} + d_{2} d_{3} + d_{2} d_{4} + d_{3} d_{4} - 4 d_{1} - 4 d_{2} - 4 d_{3} - 4 d_{4} + 10\\
	c_3&=d_{1} d_{2} d_{3} + d_{1} d_{2} d_{4} + d_{1} d_{3} d_{4} + d_{2} d_{3} d_{4} - 4 d_{1} d_{2} - 4 d_{1} d_{3} - 4 d_{1} d_{4} - 4 d_{2} d_{3} \\&=- 4 d_{2} d_{4} - 4 d_{3} d_{4} + 10 d_{1} + 10 d_{2} + 10 d_{3} + 10 d_{4} - 20\\
	c_4&=d_{1} d_{2} d_{3} d_{4} - 4 d_{1} d_{2} d_{3} - 4 d_{1} d_{2} d_{4} - 4 d_{1} d_{3} d_{4} - 4 d_{2} d_{3} d_{4} + 10 d_{1} d_{2} + 10 d_{1} d_{3}\\&+ 10 d_{1} d_{4} + 10 d_{2} d_{3} + 10 d_{2} d_{4} + 10 d_{3} d_{4} - 20 d_{1} - 20 d_{2} - 20 d_{3} - 20 d_{4} + 35\\
	s_0&=d_1d_2d_3d_4\\
	s_1&=c_1s_0\\
	s_2&=c_1^2s_0\\
	s_3&=c_1^3s_0\\
	s_{01}&=c_2s_0\\
	s_{11}&=c_1c_2s_0\\
	s_{001}&=c_3s_0.
\end{align*} Since the discrete mono and multi-singularities in this pair of dimensions are $A_1^4$, $A_1^2A_2$, $A_1A_3$, $A_2^2$, $A_4$ and $I_{2,2}$, and their Thom polynomials are already known due to works from Kazarian in \cite{9}, we can apply Ohmoto's technique from \cite{7} to obtain the formulas written below: \begin{align*}
	\# A_1^4&=\frac{1}{24}(s_1^3c_1-12s_1s_2c_1+40s_3c_1-6s_1s_{01}c_1+56s_{11}c_1+24s_{001}c_1-12s_1^2c_1^2\\&+48s_2c_1^2+24s_{01}c_1^2+120s_1c_1^3-672c_1^4-6s_1^2c_2+24s_2c_2+12s_{01}c_2\\&+168s_1c_1c_2-1776c_1^2c_2-288c_2^2+72s_1c_3-1584c_1c_3-720c_4)\\
	\# A_1^2A_2&=\frac{1}{2}(s_2c_1+s_{01}c_1-6c_1^3-12c_1c_2-6c_3)\\
	\# A_1A_3&=s_3c_1+3s_{11}c_1+2s_{001}c_1-8c_1^4-36c_1^2c_2-8c_2^2-44c_1c_3-24c_4\\
	\# A_2^2&=\frac{1}{2}(s_2c_1^2+s_{01}c_1^2-9c_1^4+s_2c_2+s_{01}c_2-36c_1^2c_2-12c_2^2-39c_1c_3-24c_4)\\
	\# A_4&=c_1^4+6c_1^2c_2+2c_2^2+9c_1c_3+6c_4\\
	\# I_{2,2}&=c_2^2-c_1c_3
\end{align*}

\section{Conclusion}\label{sec13}

In this study, we proved that $\mathcal{A}$-finite determination is a generic condition for complex homogeneous polynomial mappings from $\mathbb{C}^4$ to $\mathbb{C}^4$ with bounded degrees when considering the Zariski Topology, in contrast to the $\mathcal{C}^\infty$-Whitney Topology.

Additionally, we provide relatively explicit conditions, both necessary and sufficient, to guarantee the $\mathcal{A}$-finite determination of a generic homogeneous polynomial mapping germ. Finally, there are enumerative formulas for their singularities, depending only on the degrees $d_i$.

\backmatter

\bmhead{Acknowledgments}

The authors wish to extend their heartfelt appreciation to FAPESP (São Paulo Research Foundation) for their invaluable partial support, which played an important role in the advancement of this scientific research. The first and third authors have immensely benefited from grant 2019/21181-0, while the second author received important support through grants 2020/14442-9 and 2023/03086-5. The funding provided by FAPESP has been fundamental in the successful execution of this research, and we are deeply grateful for their continued commitment to aiding scientific progress.
\bibliography{sn-bibliography}

\end{document}